\documentclass[11pt,AutoFakeBold]{article}
       \usepackage{amsfonts}
       \usepackage{stmaryrd}
       \usepackage{latexsym,amssymb,mathrsfs,fancyhdr}
       \font\tenmsb=msbm10
       \font\sevenmsb=msbm7
       \font\fivemsb=msbm5
       \catcode`\@=11
       \ifx\amstexloaded@\relax\catcode`\@=\active
       \endinput\else\let\amstexloaded@\relax\fi
       \def\spaces@{\space\space\space\space\space}
       \def\spaces@@{\spaces@\spaces@\spaces@\spaces@\spaces@}
       \def\space@.  {\futurelet\space@\relax}
       \space@.   %
       \def\Err@#1{\errhelp\defaulthelp@\errmessage{AmS-TeX error: #1}}
       \def\relaxnext@{\let\next\relax}
       \def\accentfam@{7}
       \def\noaccents@{\def\accentfam@{0}}
       \def\Cal{\relaxnext@\ifmmode\let\next\Cal@\else
       \def\next{\Err@{Use \string\Cal\space only in math mode}}\fi\next}
       \def\Cal@#1{{\Cal@@{#1}}}
       \def\Cal@@#1{\noaccents@\fam\tw@#1}
       
       \def\Bbb{\relaxnext@\ifmmode\let\next\Bbb@\else
       \def\next{\Err@{Use \string\Bbb\space only in math mode}}\fi\next}
       \def\Bbb@#1{{\Bbb@@{#1}}}
       \def\Bbb@@#1{\noaccents@\fam\msbfam#1}
       \newfam\msbfam
       \textfont\msbfam=\tenmsb
       \scriptfont\msbfam=\sevenmsb
       \scriptscriptfont\msbfam=\fivemsb
\usepackage{dsfont}
\usepackage[german,english]{babel}
\usepackage{amsmath,amssymb}
\usepackage[square, comma, sort&compress, numbers]{natbib}
\usepackage{cases}
\usepackage{mathrsfs}
\usepackage{amsmath}
\usepackage{amsthm}
\usepackage{amsfonts}
\usepackage{amssymb}
\usepackage{latexsym}
\usepackage{fancyhdr}
\usepackage{float}
\usepackage{extarrows}
\usepackage{mathptmx}      % use Times fonts if available on your TeX system
\usepackage{amssymb,amsmath,amsfonts,enumerate,latexsym}
\numberwithin{equation}{section}
\newtheorem{thm}{Theorem}[section]

\newtheorem{lem}[thm]{Lemma}

\newtheorem{iteration lemma}[thm]{iteration Lemma}

\newtheorem{defn}[thm]{Definition}
\newtheorem{eg}[thm]{Example}
\newtheorem*{con*}{Concluding remarks}
\newtheorem*{acknowledgements*}{ACKNOWLEDGEMENTS}

\usepackage{float}
\usepackage{appendix}
\begin{document}

\setlength{\columnsep}{5pt}
%\title{\bf Solving fuzzy linear systems using a block representation of generalized inverses: The core-EP inverse}
\title{\bf Fuzzy linear systems and core-EP inverses}
\author{\ Yuefeng Gao \footnote{Corresponding author. E-mail: yfgao91@163.com }, \ Jing Li \footnote{ E-mail: jingli0204@163.com},  \\
%\ Hongxing Wang $^2$\footnote{ E-mail: wanghongxing0902@163.com }.\\
 College of Science, University of Shanghai for Science and Technology\\ Shanghai 200093,  China}
%~$^2$ School of Mathematics and Physics, Guangxi University for Nationalities\\ Nanning 530006, China
\date{}
\maketitle
\begin{quote}
\textbf {Abstract:}  {\small The main purpose of this paper is to provide a solution of the consistent fuzzy linear system and a generalized solution of the inconsistent fuzzy linear system involving the core-EP inverse of an associated matrix. 
Before this can be achieved, it is necessary to study the block structure of the core-EP inverse. Finally, results are illustrated with some numerical examples.}\\
%Firstly, we study the block structure of core-EP inverse. Then the fuzzy solutions to the fuzzy linear systems are expressed by using the core-EP inverse of the associated matrix. Secondly, the conditions for existing strong fuzzy solutions are given. In addition, some examples are presented to illustrate the proposed method.}\\
\end{quote}
\section{Introduction}%%%%%%%%%%%%%%%%%%%%%%%%%%%%%%%%%%%%%%%%%%%%%%%%%%%%%%%%%%%%%%%%%%%%%%%%%%%%%%%%%%%%%%%%%%%%%%
\label{intro}
\quad \quad The field of fuzzy linear systems have been developed rapidly since the appearance of fuzzy numbers  in \cite{Fuzzy number}. In many applications, the system's parameters are represented by fuzzy numbers rather than crisp numbers, so it is important to solve fuzzy linear systems. There are numerous papers devoted to the investigation of solutions of fuzzy linear systems \cite{{Numerical solution}, {Inconsistent fuzzy matrix equations}, {W-weighted Drazin}}. %Fuzzy linear system $A \tilde{X}=\tilde{Y}$, where the coefficient matrix $A$ is a real matrix, $\tilde{X}$ is an unknown fuzzy number vector and $\tilde{Y}$ is a given fuzzy number vector, is an extension of the crisp linear system. 

Block structures of generalized inverses such as the Moore-Penrose inverse, group inverse, $W$-weighted Drazin inverse and core inverse contributed to solving fuzzy linear systems, see \cite{M_P inverse solve, group inverse solve, core inverse solve} for example. More specifically, the method of using the Moore-Penrose inverse to solve fuzzy linear systems is firstly presented in \cite{M_P inverse solve}. Next, B. Mihailovi$\rm \acute c$ et al. \cite{group inverse solve}  got general solutions of fuzzy linear systems using the block structures of group inverses. M. Nikuie and M.Z. Ahmad \cite{W-weighted Drazin} explained the effect of $W$-weighted Drazin inverses in solving singular fuzzy linear systems. Recently, H. Jiang et al. \cite{core inverse solve} gave a method for solving fuzzy linear systems using the block structures of core inverses.
 It is well known that the core-EP inverse introduced by  K.M. Prasad and K.S. Mohana in \cite{core-EP inverse} is a generalization of the core inverse and it can be calculated by core-EP decomposition in \cite{Core-EP decomposition}.
It makes sense that solving fuzzy linear systems can be broadened via using core-EP inverses.

The paper is organized as follows: in Section 2, the definitions of fuzzy linear systems and several types of generalized inverses have arisen. In addition, we present an algorithm for computing the core-EP inverse.
 In Section 3,  we shall focus our attention on the block structure of the core-EP inverse of an associated matrix, which affords a better insight into getting a solution of the fuzzy linear system. 
 %Moreover,   block expressions of the DMP inverse and BT inverse are discussed. 
 In Section 4, We divide the problems into consistent and inconsistent fuzzy linear systems. Based on the obtained theoretical results, methods for obtaining a solution of the fuzzy linear system are proposed. In Section 5, numerical examples about dealing with consistent and inconsistent fuzzy linear systems are implemented.
 %%%%%%%%%%%%%%%%%%%%%%%%%%%%%%%%%%%%%%%%%%%%%%%%%%%%%%%%%%%%%%%%%%%%%%%%%%%%%%%%%%
\section{Preliminaries}
\label{sec:1}
\quad \quad In this section, we shall review some notations, definitions and results which play important roles in the rest sections.
\subsection{Fuzzy number and fuzzy linear system}
\label{sec:2}
\qquad Brief definitions and theorems related to fuzzy numbers and fuzzy linear systems are given in this part. The notation $X^T$ means the transpose of a vector $X$ in the following content.
\begin{defn} \cite{Fuzzy Set Theory}
\label{fuzzy number}
We represent an arbitrary fuzzy number $\tilde{z}(r)$, in parametric form, by an ordered pair of functions $\tilde{z}(r)=(\underline{z}(r), \overline{z}(r))$, $r \in[0,1]$, which satisfies the following requirements:\\
{\normalsize\textcircled{\footnotesize{1}}}
$\underline{z}(r)$ is a bounded left continuous nondecreasing function over $[0,1]$,\\
{\normalsize\textcircled{\footnotesize{2}}} 
$\overline{z}(r)$ is a bounded left continuous nonincreasing function over $[0,1]$,\\
{\normalsize\textcircled{\footnotesize{3}}}
$\underline{z}(r) \leq \overline{z}(r)$.
\end{defn}
For each  real number $\lambda,$ the scalar multiplication and the addition of fuzzy numbers can be described as follows:\\
{\normalsize\textcircled{\footnotesize{1}}} $\tilde{z} (r)+\tilde{w}(r)=(\ \underline{z}(r)+\ \underline{w}(r), \ \overline{z}(r)+\ \overline{w}(r)) ,$\\
{\normalsize\textcircled{\footnotesize{2}}} $\lambda \tilde{z}(r)=\left\{\begin{array}{ll}
{( \lambda \underline{z}(r), \lambda \overline{z}(r) ),} & \lambda \geq 0, \\
{( \lambda \overline{z}(r), \lambda \underline{z}(r) ),} & \lambda<0,
\end{array}\right.$\\
{\normalsize\textcircled{\footnotesize{3}}} $\tilde{z}(r)=\tilde{w}(r)$ if and only if $\ \underline{z}(r)=\ \underline{w}(r)$ and $\ \overline{z}(r)=\ \overline{w}(r).$\\
\begin{defn} \cite{Fuzzy linear systems}
The fuzzy linear matrix system $A \tilde{X}=\tilde{Y}$
\begin{equation}\label{FLS}
\left[\begin{array}{cccc}
a_{11} & a_{12} & \cdots & a_{1 n} \\
a_{21} & a_{22} & \cdots & a_{2 n} \\
\cdots & \cdots & \cdots & \cdots \\
a_{n 1} & a_{n 2} & \cdots & a_{n n}
\end{array}\right]\left[\begin{array}{c}
\tilde{x}_{1}\\
\tilde{x}_{2} \\
\cdots \\
\tilde{x}_{n}
\end{array}\right]=\left[\begin{array}{c}
\tilde{y}_{1} \\
\tilde{y}_{2}\\
\cdots \\
\tilde{y}_{n}
\end{array}\right],
\end{equation}
where the matrix $A=\left[a_{i j}\right]$ is a real matrix and $\tilde{x}_{i} , \tilde{y}_{i}$, $i=1,...,n$, are fuzzy numbers, is called a fuzzy linear system (FLS).
\end{defn}
\begin{defn} \cite{Fuzzy linear systems}
A fuzzy number vector $\tilde{X}(r)=\begin{bmatrix}
\tilde{x}_{1}(r), \tilde{x}_{2}(r),\cdots, \tilde{x}_{n}(r)
\end{bmatrix}
^{\rm T}$, where
$$\tilde{x}_{i}(r)=(\ \underline{x}_{i}(r), ~\ \overline{x}_{i}(r)), ~i=1,...,n,~r \in[0,1],$$
is a solution of FLS (\ref{FLS}) if it satisfies
$$\left\{\begin{array}{l}
\sum_{j=1}^{n} {a}_{i j}\ \overline{x}_{j}(r)=\ \overline{y}_{i}(r)\\
\sum_{j=1}^{n} {a}_{i j}\ \underline{x}_{j}(r)=\ \underline{y}_{i}(r)
\end{array}\right. i=1,\cdots,n.
$$\end{defn}
\indent An significant fact was noted in \cite{Fuzzy linear systems}, in order to get a solution of the FLS $A \tilde{X}=\tilde{Y}$ (\ref{FLS}), it is sufficient to solve the following crisp linear system:\\
\begin{equation}\label{SX=Y}
S X(r)=Y(r),~r \in[0,1], \end{equation}
 i.e. $$\left[\begin{array}{cccc}s_{11} & s_{12} & \cdots & s_{1,2 n} \\s_{21} & s_{12} & \cdots & s_{2,2 n} \\\vdots & \vdots & \ddots & \vdots \\s_{2 n, 1} & s_{2 n, 2} & \cdots & s_{2 n, 2 n}\end{array}\right]\left[\begin{array}{l}\ \underline{x}_{1 }(r) \\\vdots \\\underline{x}_{n }(r) \\-\ \overline{x}_{1 }(r) \\\vdots \\-\ \overline{x}_{n }(r)\end{array}\right]=\left[\begin{array}{l}\underline{y}_{1 }(r) \\\vdots \\\underline{y}_{n }(r) \\-\ \overline{y}_{1 }(r) \\\vdots \\-\ \overline{y}_{n }(r)\end{array}\right],$$
 where $s_{ij}$ are determined as follows:\\ 
 %$1 \leqslant i \leqslant 2 n$ and $1 \leqslant j \leqslant 2 n$ 
 $$
\begin{array}{l}
a_{i j} \geqslant 0 \Rightarrow s_{i j} =a_{i j},\quad s_{i+n, j+n}=a_{i j}. \\
a_{i j}<0 \Rightarrow  s_{i, j+n}=-a_{i j},\quad s_{i+n, j}=-a_{i j}.
\end{array}
$$and all the remaining $s_{ij}$ are taken zero.\\
\indent We use the next notation to represent the structure of $S$: $$S=\left[\begin{array}{ll}D & E \\ E &D \end{array}\right],$$ where $D$ and $E$ are $n\times n $ matrices, $D=\left[a_{i j}^{+}\right]$, $E=\left[a_{i j}^{-}\right]$, $a_{i j}^{+}=a_{i j} \vee 0$ and $a_{i j}^{-}=-a_{i j} \vee 0$. In this case, $S$ is called the associated matrix of $A$. We observe that $A=A^{+}-A^{-}=D-E$ and $|A|=A^{+}+A^{-}=D+E$.\\ 
%if the representative vector $X(r)$ exists and be given by: $$X(r)=\begin{bmatrix}\underline{x}_{1}(r),\  \ldots, \ \underline{x}_{n}(r), -\ \overline{x}_{1}(r), \ \ldots, -\ \overline{x}_{n}(r)\end{bmatrix}^{\rm T}$$ is a solution of crisp linear system (\ref{SX=Y}), then the correlated fuzzy number vector $\tilde{X}(r)$ must be a solution of the FLS (\ref{FLS}).\\
%\begin{defn} \cite{Numerical solution}\label{the definition of strong fuzzy solution}Let $\tilde{X}=\left\{(\ \underline {x}_{i}(r), \ \overline{x}_{i}(r)), i=1, \ \cdots, n  \right\}$ be a solution of $SX=Y$ (\ref{SX=Y}), the fuzzy number vector$\tilde{U}=\left\{(\ \underline{u}_{i}(r), \ \overline{u}_{i}(r)) , i=1,\ \cdots,n \right\}$ defined by $$\begin{array}{l}\ \underline{u}_{i}(r)=\min \left\{\ \underline{x}_{i}(r), \ \overline{x}_{i}(r), \ \underline{x}_{i}(1), \ \overline{x}_{i}(1)\right\}, \\\ \overline{u}_{i}(r)=\max \left\{\ \underline{x}_{i}(r), \ \overline{x}_{i}(r), \ \underline{x}_{i}(1), \ \overline{x}_{i}(1)\right\},\end{array}$$is called a fuzzy solution of  $SX=Y$ (\ref{SX=Y}). \end{defn}\indent If $(\ \underline{x}_{i}(r), \ \overline{x}_{i}(r)), i=1, ..., n$, are all fuzzy numbers, then $$\ \underline{x}_{i}(r)=\ \underline{u}_{i}(r),\  \overline{x}_{i}(r)=\ \overline{u}_{i}(r), i=1, \ \cdots, n,$$ in this case, $\tilde{U}$ is called a strong fuzzy solution. Otherwise, $\tilde{U}$ is a weak fuzzy solution.\\
\indent According to \cite{Inconsistent fuzzy matrix equations}, we get the following conclusions:\\
1. $\operatorname{rank}(S)<\operatorname{rank}(S \mid Y)$, $SX=Y$ does not have any solution, $A \tilde{X}=\tilde{Y}$ is called an inconsistent FLS.\\
2. $\operatorname{rank}(S)=\operatorname{rank}(S \mid Y)$, $SX=Y$  has a solution, $A \tilde{X}=\tilde{Y}$ is called a consistent FLS. Furthermore, 

(i) $\operatorname{rank}(S)=\operatorname{rank}(S \mid Y)=2n,$ $SX=Y$ has the unique solution;

(ii) $\operatorname{rank}(S)=\operatorname{rank}(S \mid Y)<2n,$ $SX=Y$ has infinite solutions.
\subsection{ core-EP inverse}%%%%%%%%%%%%%%%%%%%%%%%%%%%%%%%%%%%%%%%%%%%%%%%%%%%%%%%%%%%
\label{sec:3}
\quad \quad While the original core inverse is restricted to matrices of index one, the core-EP inverse exists for any square matrices. That is to say, it extends the core inverse of a matrix from index one to an arbitrary index. We begin with recalling some related definitions. As usual, $A^{*}$ denotes the transpose of the matrix $A$. The index of matrix $A \in \mathbb{R}^{n\times n},$ denoted by $ind(A)=k$, is the smallest nonnegative integer $k$ such that $\operatorname{rank}(A^{k+1})=\operatorname{rank}(A^{k}).$

Given $A\in \mathbb{R}^{n \times n}$, the matrix $X$ satisfying one or more of the following matrix equations has been studied extensively.
\begin{equation*}
\begin{aligned}
&(1) ~A X A=A;%\qquad (1^{\prime}) ~X A^{2}=X;
\qquad(1^{k})~ XA^{k+1}=A^{k};\\
&(2)~ X A X=X;\qquad (2^{\prime})~ A X^{2}=X;\\
&(3) ~(A X)^{*}=A X;\quad(4) ~(X A)^{*}=X A.%\quad(5)~A X=X A.
\end{aligned}
\end{equation*}

\begin{defn} 
\label{defn}
For a matrix $A\in \mathbb{R}^{n \times n}$ with index  $k$.\\
%(1) $X$ satisfying the conditions $(i)$ is denoted by $X=A^{(i)}$.\\
$(\rm i)$ \cite{M P inerse} 
$X$ is the Moore-Penrose inverse of $A$ if and only if $X$ satisfies $(1)$, $(2)$, $(3)$, $(4)$, denoted by $X=A^{\dagger}$.\\
$(\rm ii)$  \cite{core inverse} 
$X$ is the core inverse of $A$ if and only if $X$ satisfies $(1)$, $(2)^{\prime}$, $(3)$, denoted by $X=A^{\scriptsize\textcircled{\tiny\#}}$.\\
$(\rm iii)$ \cite{core-EP inverse} 
$X$ is the core-EP inverse of $A$ if and only if $X$ satisfies $(1^k)$, $(2)^{\prime},$ $(3)$, denoted by $X=A^{\scriptsize\textcircled{\tiny\dag}}$.
\end{defn}
Here, we describe a method for computing the core-EP inverse by applying the core-EP decomposition \cite{Core-EP decomposition}. Let  $A \in \mathbb{R}^{n \times n},$ $ind(A)=k,$ then there exists the unitary matrix $U$ such that
 $$A=U\left[\begin{array}{ll}T & S \\ 0 & N\end{array}\right] \boldsymbol{U}^{*},$$ 
 where $T$ is nonsingular and $N^k=0$. The core-EP inverse of $A$ has the following form:
\begin{equation}
\label{core-EP inverse µ}
\boldsymbol{A}^{\scriptsize\textcircled{\tiny\dag}}=\boldsymbol{U}\left[\begin{array}{cc}\boldsymbol{T}^{-1} & 0 \\ 0 & 0\end{array}\right] \boldsymbol{U}^{*}.
\end{equation}
We present an algorithm for computing the core-EP inverse of $A,$ when $ind(A)=k,$ as follows:
\begin{table}[H]\caption{Computation of the core-EP inverse}\label{tab:1}       \begin{tabular}{lll}\hline\textbf{Algorithm:} Computation of ${A}^{\scriptsize\textcircled{\tiny\dag}}$: $ind(A)=k$~~~~~~~~~~~~~~~~~~~~~~~~~~~~~~~~~\\\hline1. Input: $A$ is a $n \times n$ matrix;\\
2. Calculate sub-matrices $U$, $T,$ $S$ and $N$;\\
3. Calculate  matrix $T^{-1}$;\\
4. Determine the core-EP inverse using (\ref{core-EP inverse µ}). \\\hline\end{tabular}\end{table}
It is worth mentioning that the core-EP inverse ${A}^{\scriptsize\textcircled{\tiny\dag}}$ can also be expressed as (see \cite{core-EP inverse}):
 \begin{equation}
 \label{core-EP inverse jisuan}
A^{\scriptsize\textcircled{\tiny\dag}}=A^{k}[(A^{*})^{k} A^{k+1}]^{\dagger} (A^{*})^{k}.
\end{equation}
%%%%%%%%%%%%%%%%%%%%%%%%%%%%%%%%%%%%%%%%%%%%%%%%%%%%%%%%%%%%%%%%%%%%%%%%%%%%%%%%%%%%%%%%%%%%%%%%%%%%%%%%%%%%%%%%%
\section{Block structure of core-EP inverse of the associated matrix $S$}
\quad\quad The block structure of $S^{\scriptsize\textcircled{\tiny\dag}}$, which we shall consider in this part, is a powerful tool in solving FLS. Provided that the structure of the associated matrix $S$ is $\left[\begin{array}{ll}
D & E \\
E & D
\end{array}\right],$ where $D, E\in \mathbb{R}^{n \times n}$.
\begin{lem}
\label{An}
Let $S$ be an arbitrary $ 2n\times 2n $  matrix with the form\\
$$S=\left[\begin{array}{ll}
D & E \\
E & D
\end{array}\right],$$
where both $D$ and $E$  are $ n\times n $ matrices, then
$$S^{n}=\left[\begin{array}{cc}
\frac{1}{2}(D+E)^{n}+\frac{1}{2}(D-E)^{n} & \frac{1}{2}(D+E)^{n}-\frac{1}{2}(D-E)^{n}\\
\frac{1}{2}(D+E)^{n}-\frac{1}{2}(D-E)^{n} & \frac{1}{2}(D+E)^{n}+\frac{1}{2}(D-E)^{n}
\end{array}\right].$$
\end{lem}
\begin{proof}
%Let $P=\left[\begin{array}{ll}1 & -1 \\1 & -1\end{array}\right],$ then $P^{-1}=\left[\begin{array}{cc}\frac{1}{2} & \frac{1}{2} \\\frac{1}{2} &\frac{1}{2}\end{array}\right].$  Therefore, $$P^{-1} W P=\left[\begin{array}{cc}B + C & 0 \\ 0 & B+C\end{array}\right]:=\Lambda.$$  
This result can be proved by mathematical induction.\\
\indent Firstly, \\
\begin{equation*}
\begin{aligned}
S^{2}&=\left[\begin{array}{ll}D^{2}+E^{2} & D E+E D \\ D E+E D & D^{2}+E^{2}\end{array}\right]\\&=
\left[\begin{array}{ll}\frac{1}{2}(D+E)^{2}+\frac{1}{2}(D-E)^{2} & \frac{1}{2}(D+E)^{2}-\frac{1}{2}(D-E)^{2} \\
\frac{1}{2}(D+E)^{2}-\frac{1}{2}(D-E)^{2}  & \frac{1}{2}(D+E)^{2}+\frac{1}{2}(D-E)^{2} \end{array}\right].
\end{aligned}
\end{equation*}
If the situation of $n=k$ is true, then we have
$$\begin{aligned}
S^{k} &=\left[\begin{array}{cc}
\frac{1}{2}(D+E)^{k}+\frac{1}{2}(D-E)^{k} & \frac{1}{2}(D+E)^{k}-\frac{1}{2}(D-E)^{k} \\
\frac{1}{2}(D+E)^{k}-\frac{1}{2}(D-E)^{k} & \frac{1}{2}(D+E)^{k}+\frac{1}{2}(D-E)^{k}
\end{array}\right].
\end{aligned}$$
When $n=k+1$,
$$\begin{aligned}S^{k+1}=\left[\begin{array}{ll}D & E \\ E & D\end{array}\right]\left[\begin{array}{cc}\frac{1}{2}(D+E)^{k}+\frac{1}{2}(D-E)^{k} & \frac{1}{2}(D+E)^{k}-\frac{1}{2}(D-E)^{k} \\
\frac{1}{2}(D+E)^{k}-\frac{1}{2}(D-E)^{k} & \frac{1}{2}(D+E)^{k}+\frac{1}{2}(D-E)^{k}\end{array}\right]\\=\left[\begin{array}{ll}\frac{1}{2}(D+E)^{k+1}+\frac{1}{2}\left(D-E\right)^{k+1} &\frac{1}{2}(D+E)^{k+1} -\frac{1}{2}(D-E)^{k+1} \\ \frac{1}{2}(D+E)^{k+1}-\frac{1}{2}(D-E)^{k+1} & \frac{1}{2}(D+E)^{k+1}+\frac{1}{2}(D-E)^{k+1}\end{array}\right].\end{aligned}$$
It completes the proof.
\end{proof}
We deal with the block structure of the core-EP inverse in the next statement.
\begin{thm}
\label{the block structure of core-EP inverse}
%%%%%%%%%%%%%%%%%%%%%%%%%%%%%%%%%%%%%%%%%%%%%%%%%%%%%%%%%%%%%%%%%core-ep
Let $A$ be the coefficient matrix of FLS and $S$ be its associated matrix. The core-EP inverse $S^{\scriptsize\textcircled{\tiny\dag}}$ of $S$ is
\begin{equation}
\label{core-ep de jiegou }
S^{\scriptsize\textcircled{\tiny\dag}}=\left[\begin{array}{ll}
H & Z \\
Z & H
\end{array}\right]
\end{equation} 
\label{core-ep inverse }
if and only if $$\begin{array}{l}
H=\frac{1}{2}\left[(D+E)^{\scriptsize\textcircled{\tiny\dag}}+(D-E)^{\scriptsize\textcircled{\tiny\dag}}\right], ~Z=\frac{1}{2}\left[(D+E)^{\scriptsize\textcircled{\tiny\dag}}-(D-E)^{\scriptsize\textcircled{\tiny\dag}}\right]
\end{array}.$$
\end{thm}
\begin{proof}
%Observe that $A=A^{+}-A^{-}=D-E$ and $|A|=A^{+}+A^{-}=D+E$.\\
$(\Rightarrow)$
Firstly, since, as the concept of the core-EP inverse stated, $S^{\scriptsize\textcircled{\tiny\dag}}S^{k+1}=S^{k}$, \\
then,
$$\left[\begin{array}{ll}
H & Z \\
Z & H
\end{array}\right]\left[\begin{array}{ll}
D & E \\
E & D
\end{array}\right]^{k+1}=\left[\begin{array}{ll}
D & E \\
E & D
\end{array}\right]^{k}.$$
So,
$$\left[\begin{array}{ll}
H D+Z E & Z D+H E \\ 
Z D+H E & H D+Z E
\end{array}\right]\left[\begin{array}{ll}
D & E \\
E & D
\end{array}\right]^{k}\\=\left[\begin{array}{ll}
D & E \\
E & D
\end{array}\right]^{k}.$$
In view of Lemma \ref{An}, we claim that
$$\begin{array}{ll}
{\left[\begin{array}{cc}
H D+Z E & Z D+H E \\
Z D+H E & H D+Z E
\end{array}\right]\left[\begin{array}{cc}
\frac{1}{2}(D+E)^{k}+\frac{1}{2}(D-E)^{k} & \frac{1}{2}(D+E)^{k}-\frac{1}{2}(D-E)^{k} \\
\frac{1}{2}(D+E)^{k}-\frac{1}{2}(D-E)^{k} & \frac{1}{2}(D+E)^{k}+\frac{1}{2}(D-E)^{k}
\end{array}\right]} \\
=\left[\begin{array}{ll}
\frac{1}{2}(D+E)^{k}+\frac{1}{2}(D-E)^{k} & \frac{1}{2}(D+E)^{k}-\frac{1}{2}(D-E)^{k} \\
\frac{1}{2}(D+E)^{k}-\frac{1}{2}(D-E)^{k} & \frac{1}{2}(D+E)^{k}+\frac{1}{2}(D-E)^{k}
\end{array}\right], 
\end{array}$$
which guarantees,\\
$$\begin{aligned}
&(Z D+H E)[(D+E)^{k}+(D-E)^{k}]+(H D+Z E)[(D+E)^{k}-(D-E)^{k}]\\
&=(D+E)^{k}-(D-E)^{k},
\end{aligned}$$
$$\begin{aligned}
&(H D+Z E)[(D+E)^{k}+(D-E)^{k}]+(Z D+H E)[(D+E)^{k}-(D-E)^{k}]\\
&=(D+E)^{k}+(D-E)^{k}.
\end{aligned}$$
This amounts to the two formulas, 
$$(H+Z)(D+E)^{k+1}+(H-Z)(D-E)^{k+1}=(D+E)^{k}+(D-E)^{k},$$
$$(H+Z)(D+E)^{k+1}-(H-Z)(D-E)^{k+1}=(D+E)^{k}-(D-E)^{k}.$$
Clearly these mean,\\
$$(H+Z)(D+E)^{k+1}=(D+E)^{k}, ~(H-Z)(D-E)^{k+1}=(D-E)^{k},$$
i.e. $$H+Z=(D+E)^{(1^{k})} ,~ H-Z=(D-E)^{(1^{k})}.$$
Hence,
$$H=\frac{1}{2}\left[(D+E)^{(1^{k})}+(D-E)^{(1^{k})}\right] ,~ Z=\frac{1}{2}\left[(D+E)^{(1^{k})}-(D-E)^{(1^{k})}\right].$$

Secondly, from $S(S^{\scriptsize\textcircled{\tiny\dag}})^{2}=S^{\scriptsize\textcircled{\tiny\dag}},$
it follows that $$\left[\begin{array}{ll}
D & E \\
E & D
\end{array}\right]\left[\begin{array}{ll}
H & Z \\
Z & H
\end{array}\right]^{2}=\left[\begin{array}{ll}
H & Z \\
Z & H
\end{array}\right].$$
It gives\\
$$\left[\begin{array}{ll}
(D H + E Z) H+(D Z+E H) Z & (D H+E Z) Z+(D Z+E H)H  \\
(D H+E Z) Z+(D Z+E H)H &(D H + E Z) H+(D Z+E H) Z
\end{array}\right]=\left[\begin{array}{ll}
H & Z \\
Z & H
\end{array}\right].$$
Thus $$H=(D H+E Z) H+(D Z+E H) Z ,~ Z=(D H+E Z) Z+(D Z+E H) H.$$
This shows that, $H+Z=(D+E)(H+Z)^{2} ,~H-Z=(D-E)(H-Z)^{2},$\\
i.e. $$H+Z=(D+E)^{(2^{\prime})} ,~ H-Z=(D-E)^{(2^{\prime})}.$$
We conclude that
$$H=\frac{1}{2}\left[(D+E)^{(2^{\prime})}+(D-E)^{(2^{\prime})}\right] ,~ Z=\frac{1}{2}\left[(D+E)^{(2^{\prime})}-(D-E)^{(2^{\prime})}\right].$$
Lastly, according to $(S S^{\scriptsize\textcircled{\tiny\dag}})^{*}=S S^{\scriptsize\textcircled{\tiny\dag}},$
we obtain 
 $$\left[\begin{array}{ll}H & Z \\ Z & H\end{array}\right]^{*}\left[\begin{array}{ll}D & E \\ E & D\end{array}\right]^{*}=\left[\begin{array}{ll}D & E \\ E & D\end{array}\right]\left[\begin{array}{ll}H & Z \\ Z & H\end{array}\right].$$
It is not difficult to prove,
 $$(D H+E Z)^{*}=D H+E Z, ~ (D Z+E H)^{*}=D Z+E H .$$
So we have $${[(D+E)(H+Z)]^{*}=(D+E)(H+Z)},~ {[(D-E)(H-Z)]^{*}=(D-E)(H-Z)},$$
i.e.
$$ H+Z=(D+E)^{(3)},~ H-Z=(D-E)^{(3)}.$$
Clearly, 
$$H=\frac{1}{2}\left[(D+E)^{(3)}+(D-E)^{(3)}\right], ~Z=\frac{1}{2}\left[(D+E)^{(3)}-(D-E)^{(3)}\right].$$
To sum up,
$$H=\frac{1}{2}\left[(D+E)^{\scriptsize\textcircled{\tiny\dag}}+(D-E)^{\scriptsize\textcircled{\tiny\dag}}\right],~ Z=\frac{1}{2}\left[(D+E)^{\scriptsize\textcircled{\tiny\dag}}-(D-E)^{\scriptsize\textcircled{\tiny\dag}}\right].$$
$(\Leftarrow)$ 
Given
$$
H=\frac{1}{2}\left[(D+E)^{\scriptsize\textcircled{\tiny\dag}}+(D-E)^{\scriptsize\textcircled{\tiny\dag}}\right],~
Z=\frac{1}{2}\left[(D+E)^{\scriptsize\textcircled{\tiny\dag}}-(D-E)^{\scriptsize\textcircled{\tiny\dag}}\right],$$ 
%$$ H+Z=(D+E)^{\scriptsize\textcircled{\tiny\dag}},~H-Z=(D-E)^{\scriptsize\textcircled{\tiny\dag}}. $$
it is immediate that  
%So,
\begin{equation*}
\begin{aligned}
&(D+E)(H+Z)^{2}=(H+Z), ~
(D-E)(H-Z)^{2}=(H-Z), \\
&[(D+E)(H+Z)]^{*}=(D+E)(H+Z), ~
[(D-E)(H-Z)]^{*}=(D-E)(H-Z),\\
&(H+Z)(D+E)^{k+1}=(D+E)^{k}, ~
(H-Z)(D-E)^{k+1}=(D-E)^{k}.
\end{aligned}
\end{equation*}
First of all, from $$
\begin{aligned}
H &=\frac{1}{2}(H+Z)+\frac{1}{2}(H-Z) \\
&=\frac{1}{2}(D+E)(H+Z)^{2}+\frac{1}{2}(D-E)(H-Z)^{2} \\
&=(D H+E Z) H+(D Z+E H) Z,
\end{aligned}$$
$$
\begin{aligned}
Z&=\frac{1}{2}(H+Z)-\frac{1}{2}(H-Z) \\
&=\frac{1}{2}(D+E)(H+Z)^{2}-\frac{1}{2}(D-E)(H-Z)^{2} \\
&=(D H+E Z) Z+(D Z+E H) H,
\end{aligned}
$$
it follows that$$
\left[\begin{array}{ll}
D & E \\
E & D
\end{array}\right]\left[\begin{array}{ll}
H & Z \\
Z & H
\end{array}\right]^{2}=\left[\begin{array}{ll}
H & Z \\
Z & H
\end{array}\right].$$
Secondly, note that \begin{align*}
D H+E Z &=\frac{1}{2}(D+E)(H+Z)+\frac{1}{2}(D-E)(H-Z) \\
&=\frac{1}{2}(H+Z)^{*}(D+E)^{*}+\frac{1}{2}(H-Z)^{*}(D-E)^{*} \\
&=(D H+E Z)^{*},\\
E H+D Z &=\frac{1}{2}(D+E)(H+Z)-\frac{1}{2}(D-E)(H-\mathrm{Z}) \\
&=\frac{1}{2}(H+Z)^{*}(D+E)^{*}-\frac{1}{2}(H-\mathrm{Z})^{*}(D-E)^{*} \\
&=(E H+D Z)^{*}.
\end{align*}
So we derive that $$
\left[\begin{array}{ll}
H & Z \\
Z & H
\end{array}\right]^{*}\left[\begin{array}{ll}
D & E \\
E & D
\end{array}\right]^{*}=\left[\begin{array}{ll}
D & E \\
E & D
\end {array}\right]\left[\begin{array}{ll}
H & Z \\
Z & H
\end {array}\right]
.$$
At last, we find that\\
$$
(H+Z)(D+E)^{k+1}=(D+E)^{k}, ~
(H-Z)(D-E)^{k+1}=(D-E)^{k}.\\
$$
So, $$\begin{array}{l}
(H+Z)(D+E)^{k+1}-(H-Z)(D-E)^{k+1}=(D+E)^{k}-(D-E)^{k} ,\\
(H+Z)(D+E)^{k+1}+(H-Z)(D-E)^{k+1}=(D+E)^{k}+(D-E)^{k},
\end{array}$$
i.e.
$$\begin{aligned}
&(Z D+H E)[(D+E)^{k}+(D-E)^{k}]+(Z E+H D)[(D+E)^{k}-(D-E)^{k}]\\
&=(D+E)^{k}-(D-E)^{k},\\
&(H D+Z E)[(D+E)^{k}+(D-E)^{k}]+(H E+Z D)[(D+E)^{k}-(D-E)^{k}]\\
&=(D+E)^{k}+(D-E)^{k}.
\end{aligned}$$
Evidently,
$$\begin{aligned}
&\frac{1}{2}(H D+Z E)\left[(D+E)^{k}-(D-E)^{k}\right]+\frac{1}{2}(H E+Z D)\left[(D+E)^{k}-(D-E)^{k}\right]\\
&=\frac{1}{2}\left[(D+E)^{k}+(D-E)^{k}\right],\\
&\frac{1}{2}(Z D+H E)\left[(D+E)^{k}-(D-E)^{k}\right]+\frac{1}{2}(Z E+H D)\left[(D+E)^{k}-(D-E)^{k}\right]\\
&=\frac{1}{2}\left[(D+E)^{k}-(D-E)^{k}\right].
\end{aligned}$$
With the application of Lemma \ref{An},
$$\left[\begin{array}{cc}
H & Z \\
Z & H
\end{array}\right]
\left[\begin{array}{ll}
D & E \\
E & D
\end{array}\right]^{k+1}=
\left[\begin{array}{ll}
D & E \\
E & D
\end{array}\right]^{k}.$$
In conclusion, $$
S^{\scriptsize\textcircled{\tiny\dag}}=\left[\begin{array}{ll}
H & Z \\
Z & H
\end{array}\right]
.$$
\end{proof}

\section{Methods for giving a solution of FLS} 
\qquad In this section,  new methods involving core-EP inverses for giving a solution of FLS are proposed. For the consistent and inconsistent FLS, we will study them separately. We now state the main theorems of this paper. $\mathcal {R}(X)$ denotes the columu space of a matrix $X$.
\begin{thm}
\label{core-EP inverse is solution}
$S^{\scriptsize\textcircled{\tiny\dag} }Y$ is a solution of the crisp linear system $SX=Y$ 
if and only if $Y\in \mathcal {R}(S^{k}),$ where $k=ind(S)$.
\end{thm}
\begin{proof}
$(\Leftarrow)$ Due to $Y\in \mathcal {R}(S^{k}) ~,$ we have $Y=S^{k}T$, where $T\in \mathbb{R}^{n \times 1}$. \\
According to the definition of core-EP inverse,\\
So,$$
SS^{\scriptsize\textcircled{\tiny\dag}}Y=SS^{\scriptsize\textcircled{\tiny\dag}}S^{k}T=S^{k}T=Y.$$
It turns out that $ S^{\scriptsize\textcircled{\tiny\dag}} Y $ is a solution of $SX=Y.$ \\
$(\Rightarrow)$ Since $S^{\scriptsize\textcircled{\tiny\dag}} Y$ is a solution of $SX=Y$, that $SS^{\scriptsize\textcircled{\tiny\dag}} Y=Y$.\\
From $S^{\scriptsize\textcircled{\tiny\dag}}=S(S^{\scriptsize\textcircled{\tiny\dag}})^{2},$\\
it follows that
$$
\begin{array}{l}
Y=S S^{\scriptsize\textcircled{\tiny\dag}} Y 
=S S(S^{\scriptsize\textcircled{\tiny\dag}})^{2} Y 
=S^{2}(S^{\scriptsize\textcircled{\tiny\dag}})^{2} Y  
=\cdots
=S^{k}(S^{\scriptsize\textcircled{\tiny\dag}})^{k} Y\end{array},
$$
which leads to $ Y\in \mathcal {R}(S^{k}).$
\end{proof}
\textbf{Proposed method 1}:\\
$(a)$ If $ind(S)=0$, then the crisp linear system $SX=Y$ is consistent and $X=S^{-1}Y$ is clearly the unique solution of $SX=Y$. \\
$(b)$ If $ind(S)\neq 0$ and $Y\in \mathcal {R}(S^{k})$, then the crisp linear system $SX=Y$ is consistent and $X=S^{\scriptsize\textcircled{\tiny\dag} }Y$ is a solution of $SX=Y$, see Theorem \ref{core-EP inverse is solution}.\\
%$(b)$ If $ind(S)\geq 1$ and the system $SX=Y$ is consistent, then $X=S^{\scriptsize\textcircled{\tiny\dag}}Y$.\\
\indent %If a crisp linear system does not have the unique solution, then associated FLS does not have one either.\cite{Fuzzy linear systems}
According to \cite[Theorem 2.5 and Definition 2.5]{Inconsistent fuzzy matrix equations},  if the crisp linear system $SX=Y$ (\ref{SX=Y}) is inconsistent, then the associated FLS $A \tilde{X}=\tilde{Y}$ (\ref{FLS}) is also inconsistent. Under this circumstances, we desire to find a generalized  solution of the inconsistent FLS $A \tilde{X}=\tilde{Y}$ (\ref{FLS}) by the following two approaches. \\
% An approximation solution of the inconsistent crisp linear system $SX=Y$ (\ref{SX=Y}) can be obtained by the following two approaches. 
\indent\textbf{Proposed method 2}:\\
%We can use core-EP inverse to solve inconsistent linear system $SX=Y$:\\
$({\rm i})$ Through solving the consistent crisp linear system 
\begin{equation}
\label{jiefa1} 
S X=S^{k}(S^{k})^{(1,3)} Y.
\end{equation} By Theorem \ref{core-EP inverse is solution}, \begin{equation}X=S^{\scriptsize\textcircled{\tiny\dag}} S^{k} (S^{k})^{(1,3)} Y=S^{\scriptsize\textcircled{\tiny\dag} }Y \end{equation} is a solution of the above crisp linear system.\\
\label{core-ep jie de xingshi}
$(\rm ii)$ Through solving the consistent crisp linear system
\begin{equation}\label{zhuanzhi system}
(S^{k})^{*} S X=(S^{k})^{*} Y.\end{equation}
By Theorem \ref{core-EP inverse is solution}, \begin{equation}X=[(S^k)^{*}S]^{\scriptsize\textcircled{\tiny\dag}}(S^{k})^{*} Y=S^{\scriptsize\textcircled{\tiny\dag} }Y\end{equation} is a solution of the above crisp linear system.
\begin {defn}
For the inconsistent FLS $A \tilde{X}=\tilde{Y}$ (\ref{FLS}), $X=S^{\scriptsize\textcircled{\tiny\dag} }Y$ is a solution of the crisp linear system (\ref{jiefa1}) (resp. (\ref{zhuanzhi system})), then associated fuzzy number vector $\tilde{X}$ is called a generalized solution of the FLS $A \tilde{X}=\tilde{Y}$ (\ref{FLS}).
\end {defn}
\section{Numerical examples}%%%%%%%%%%%%%%%%%%%%%%%%%%%%%%%%%%%%%%%%%%%%%%%%%%%%%%%%%%%%%%%%%%%%%%%%%%%%%%%%%%%%%%
\qquad Our aim in this section is to use some examples about obtaining a solution of the consistent and inconsistent FLS to illustrate the methods presented in this paper. All the numerical tasks have been performed by using Matlab R2019b.

%%%%%%%%%%%%%%%%%%%%%%%%%%%%%%%%%%%%%%%%%%%%%%%%%%%%%%%%%%%%%%%%%%%%%%%%%%%%%%%%%%%%%%%%%%%
\begin{eg}
Consider the following $2\times 2$ order consistent FLS:
\begin{equation}
\label{(EX2)}
\left[\begin{array}{cc}
-2& 1\\
4& -2 
\end{array}\right]
\left[\begin{array}{cc}
\tilde{x}_{1} \\
\tilde{x}_{2} 
\end{array}\right]=\left[\begin{array}{cc}
(-1+3r,~3-r)  \\
(-6+2r,~2-6r) \\
\end{array}\right]
.\end{equation}
\label{example 5.2}
The extended $4 \times 4$ order crisp linear system $SX=Y$ is 
\begin{equation}
\label{(EX21)}
\left[\begin{array}{cccc}
0 & 1 & 2 & 0 \\
4 & 0 & 0 & 2 \\
2 & 0 & 0 & 1 \\
0 & 2 & 4 & 0
\end{array}\right]\left[\begin{array}{cc}
\underline{x}_{1}  \\
\underline{x}_{2}  \\
-\ \overline{x}_{1}  \\
-\ \overline{x}_{2} 
\end{array}\right]=\left[\begin{array}{cc}
-1+3r  \\
-6+2r  \\
-3+r  \\
-2+6r
\end{array}\right]
.\end{equation}
Since $\operatorname{rank}(S)=\operatorname{rank}(S \mid Y)=2$, then the crisp linear system $SX=Y$ is consistent. Note that $ind(S)=1.$\\
By \textbf{Proposed method 1},  a solution of (\ref{(EX21)}) is given by
$$X=S^{\scriptsize\textcircled{\tiny\dag}}Y=S^{\scriptsize\textcircled{\tiny\#}}Y.$$
The core inverse of $S$ is
$$S^{\scriptsize\textcircled{\tiny\#}}=\left[\begin{array}{cccc}
 0.0000  &  0.1000  &  0.0500 &  0.0000\\
    0.1000     &    0.0000    &      0.0000  &   0.2000\\
    0.0500    &     0.0000      &    0.0000  &   0.1000\\
   0.0000   & 0.2000  &  0.1000  & 0.0000\\
\end{array}\right].
$$
Finally, a solution of the consistent linear system (\ref{(EX21)}) is $$X=
\left[\begin{array}{cc}
\underline{x}_{1}  \\
\underline{x}_{2}  \\
-\overline{x}_{1}  \\
-\overline{x}_{2} 
\end{array}\right]=\left[\begin{array}{cc}
 -0.7500 + 0.2500r\\
  -0.5000 + 1.5000r\\
  -0.2500 + 0.7500r\\
  -1.5000 + 0.5000r\\
\end{array}\right],$$
then the associated fuzzy number vector\
$$\begin{aligned}\tilde{X}=
\left[\begin{array}{cc}
\tilde{x}_{1} \\
\tilde{x}_{2}\\ 
\end{array}\right]=\left[\begin{array}{cccc}
(-0.7500 + 0.2500r, ~0.2500 -0.7500r)  \\
( -0.5000 + 1.5000r, ~1.5000 -0.5000r )    \\
\end{array}\right].\end{aligned}$$
is a solution of the FLS (\ref{(EX2)}).
%According to Definition \ref{the definition of strong fuzzy solution}, the fuzzy number solution given by$$\begin{aligned}\tilde{U}=\left[\begin{array}{cc}\tilde{u}_{1} \\\tilde{u}_{2}\\ \end{array}\right]=\left[\begin{array}{cccc}(-1,~ 0.6667-0.6667r)   \\(-2 , ~ 1.3333-1.3333r)    \\\end{array}\right].\end{aligned}$$is a fuzzy solution and is a weak fuzzy solution.
%Thus $$\begin{aligned}\left[\begin{array}{cc}\tilde{x}_{11} & \tilde{x}_{12} \\\tilde{x}_{21} & \tilde{x}_{22}\end{array}\right]=\left[\begin{array}{cccc}(-0.6664 - 0.3334r, 0.6664 - 0.6666r) & (0.1667-0.8333r, 0.3333-0.6667r) \\(-1.3336 - 0.6666r, 1.3336 - 1.3334r) & (0.3333-1.6667r, 0.6667-1.3333r)  \end{aligned}.
\end{eg}
\begin{eg}
Consider the following $3\times3$ order consistent FLS:\\
\begin{equation}
\label{(EX3)}
\left[\begin{array}{ccc}
2 & 0 & 0 \\
-1 & 1 & 1 \\
-1 & -1 & -1
\end{array}\right]\left[\begin{array}{cc}
\tilde{x}_{1} \\
\tilde{x}_{2} \\
\tilde{x}_{3} 
\end{array}\right]=\left[\begin{array}{cc}
(4r,~-8)\\
(12-r,~-4-3r)\\
(8+r,~-8-r)
\end{array}\right].
\end{equation}
The extended $6 \times 6$ order crisp linear system $SX=Y$ is 
\begin{equation}
\label{(EX31)}
\left[\begin{array}{cccccc}
2 & 0 & 0 & 0 & 0 & 0\\
0 & 1 & 1 & 1 & 0 & 0\\
0 & 0 & 0 & 1 & 1 & 1\\
0 & 0 & 0 & 2 & 0 & 0\\
1 & 0 & 0 & 0 & 1 & 1\\
1 & 1 & 1 & 0 & 0 & 0
\end{array}\right]\left[\begin{array}{cccccc}
\underline{x}_{1}  \\ 
\underline{x}_{2}  \\
\underline{x}_{3}  \\
-\overline{x}_{1} \\ 
-\overline{x}_{2} \\
-\overline{x}_{3} \\
\end{array}\right]=\left[\begin{array}{cccccc}
4r     \\
12-r \\
8+r \\
8    \\
4 + 3r \\
8+r 
\end{array}\right].
\end{equation}
Note that $ind(S)=2$. We can get $$S^{2}=\left[\begin{array}{cccccc}
4 & 0 & 0 & 0 & 0 & 0\\
0 & 1 & 1 & 4 & 1 & 1\\
2 & 1 & 1 & 2 & 1 & 1\\
0 & 0 & 0 & 4 & 0 & 0\\
4 & 1 & 1 & 0 & 1 & 1\\
2 & 1 & 1 & 2 & 1 & 1
\end{array}\right].$$
Then, $Y=S^{2} \cdot\left[\begin{array}{ll}
r \\
2 \\
-5-r \\
2  \\
4  \\
3 
\end{array}\right],$ $Y\in \mathcal {R}(S^{2}).$\\
According to Theorem \ref{core-EP inverse is solution}, $X=S^{\scriptsize\textcircled{\tiny\dag} }Y$ is a solution of (\ref{(EX31)}).\\
So,
$$S^{\scriptsize\textcircled{\tiny\dag} }=\left[\begin{array}{cccccc}
 0.3750  & -0.1250 &   0.0000  &  0.1250 &   0.1250   & 0.0000\\
   -0.2500  &  0.2500  &  0.1250  &  0.0000 &  0.0000  &  0.1250\\
   -0.1250  &  0.1250  &  0.1250  & -0.1250  &  0.1250  &  0.1250\\
    0.1250  &  0.1250  & 0.0000  &  0.3750  & -0.1250  & 0.0000\\
   0.0000  &  0.0000   & 0.1250  & -0.2500  &  0.2500 &   0.1250\\
   -0.1250  &  0.1250  &  0.1250  & -0.1250  &  0.1250  &  0.1250
\end{array}\right].$$
Therefore, a solution of (\ref{(EX31)}) is
$$\begin{aligned}X=\left[\begin{array}{cccccc}
\underline{x}_{1}  \\ 
\underline{x}_{2} \\
\underline{x}_{3}  \\
 -\overline{x}_{1}  \\ 
 -\overline{x}_{2} \\
  -\overline{x}_{3} \\
 \end{array}\right]=\left[\begin{array}{cccc}
   2r \\
   5-r \\
   3 \\
  4 \\
   1 +r  \\
   3 
\end{array}\right],\end{aligned}$$
then the associated fuzzy number vector\\
$$\begin{aligned}\tilde{X}=
\left[\begin{array}{cc}
\tilde{x}_{1} \\
\tilde{x}_{2}\\
\tilde{x}_{3} 
\end{array}\right]=\left[\begin{array}{cccc}
(2r,~ -4)  \\
( 5-r,~-1-r)  \\
(3,~-3 ) 
\end{array}\right].\end{aligned}$$
is a solution of the FLS (\ref{(EX3)}).
%which implies that $\tilde{X}$ is not a fuzzy number vector, therefore the corresponding fuzzy solution is a weak fuzzy solution of (\ref{(EX3)}) given by$$\begin{aligned}\tilde{U}=\left[\begin{array}{cc}\tilde{u}_{1} \\\tilde{u}_{2}\\\tilde{u}_{3} \end{array}\right]=\left[\begin{array}{cccc}(-4,~ 2)   \\(-2,~ 5-r)    \\(-3,~ 3) \end{array}\right].\end{aligned}$$
\end{eg}
\begin{eg}
Consider the following $2\times2$ order inconsistent FLS :
\begin {equation}\label{inconsistent FLS}
\left[\begin{array}{cc}
-1 & 1 \\
-1 & 1
\end{array}\right]\left[\begin{array}{ll}
\tilde{x}_{1}  \\
\tilde{x}_{2} 
\end{array}\right]=\left[\begin{array}{cc}
(3,~ 2+r)  \\
(4, ~8r) 
\end {array}\right].
\end{equation}
The extended $4 \times 4$ order crisp linear system $SX=Y$ is 
\begin {equation}
\left[\begin{array}{cccc}\label{inconsistent SX=Y}
0 & 1 & 1 & 0 \\ 
0 & 1 & 1 & 0 \\ 
1 & 0 & 0 & 1 \\ 
1 & 0 & 0 & 1
\end{array}\right]
\left[\begin{array}{cc}
\underline{x}_{1} \\
 \underline{x}_{2}  \\
  -\ \overline{x}_{1}\\
   -\ \overline{x}_{2}
   \end{array}\right]
   =\left[\begin{array}{cc}
   3  \\ 
   4   \\ 
   -2-r \\ 
   -8r 
\end{array}\right].
\end {equation}
Since $\operatorname{rank}(S)=2,$ $\operatorname{rank}(S \mid Y)=4,$ $\operatorname{rank}(S)<\operatorname{rank}(S \mid Y)$, then the crisp linear system $SX=Y$ (\ref{inconsistent SX=Y}) is inconsistent and the FLS (\ref{inconsistent FLS}) is also inconsistent. Note that $ind(S)=2.$\\
$({\rm i})$ We consider the following consistent crisp linear system:
\begin{equation}
\label{(a)}
S X=S^{2}(S^{2})^{(1,3)} Y.
\end{equation}
%We have $$
%S=\left[\begin{array}{llll}
%0 & 1 & 1 & 0 \\
%0 & 1 & 1 & 0 \\
%1 & 0 & 0 & 1 \\
%1 & 0 & 0 & 1
%\end{array}\right],~
%S^{2}=\left[\begin{array}{llll}
%1 & 1 & 1 & 1 \\
%1 & 1 & 1 & 1 \\
%1 & 1 & 1 & 1 \\
%1 & 1 & 1 & 1
%\end{array}\right].$$
The core-EP decomposition of $S$ is
$$S=U\left[\begin{array}{ll}T & S \\ 0 & N\end{array}\right] \boldsymbol{U}^{*},$$ 
%$$S=U\left[\begin{array}{cccc}2.0000  & 0.0000  &  0.0000  &  0.0000\\0.0000  & 0.0000  &  -1.4142 &  1.4142\\0.0000  & 0.0000  &  0.0000  &  0.0000\\0.0000  & 0.0000  &  0.0000  &  0.0000\end{array}\right] U^{*},$$ 
where $$U=\left[\begin{array}{cccc}
0.5000  &  0.5000  &  0.7071 &  0.0000\\
0.5000  &  0.5000  & -0.7071  &  0.0000\\
0.5000  & -0.5000  &  0.0000  &  0.7071\\
0.5000  & -0.5000  &  0.0000  & -0.7071
\end{array}\right], T=\left[\begin{array}{ccc}
2.0000\\
\end{array}\right],$$
$$S=\left[\begin{array}{ccc}
0.0000 & 0.0000 & 0.0000 \\
\end{array}\right] and~
 N=\left[\begin{array}{ccc}
0.0000 & -1.4142 & 1.4142\\
0.0000 & 0.0000 & 0.0000\\
0.0000 & 0.0000 & 0.0000
\end{array}\right].$$
By \textbf{Algorithm}, 
$$
S^{\scriptsize\textcircled{\tiny\dag}}=\boldsymbol{U}\left[\begin{array}{cc}\boldsymbol{T}^{-1} & 0 \\ 0 & 0\end{array}\right] \boldsymbol{U}^{*}=\left[\begin{array}{cccc}
0.1250 & 0.1250 & 0.1250 & 0.1250 \\
0.1250 & 0.1250 & 0.1250 & 0.1250 \\
0.1250 & 0.1250 & 0.1250 & 0.1250 \\
0.1250 & 0.1250 & 0.1250 & 0.1250 
\end{array}\right]
.$$
According to \textbf{Proposed method 2}, a solution of (\ref{(a)}) is given by $$
X=S^{\scriptsize\textcircled{\tiny\dag}}Y
.$$ 
Therefore,
$$\begin{aligned}X&=\left[\begin{array}{cc}\underline{x}_{1} \\
\underline{x}_{2} \\ 
-\overline{x}_{1}  \\ 
-\overline{x}_{2} 
\end{array}\right]=\left[\begin{array}{cccc}
0.1250 & 0.1250 & 0.1250 & 0.1250 \\
0.1250 & 0.1250 & 0.1250 & 0.1250 \\
0.1250 & 0.1250 & 0.1250 & 0.1250 \\
0.1250 & 0.1250 & 0.1250 & 0.1250 
\end{array}\right]\left[\begin{array}{cc}
3  \\
 4 \\
-2-r \\
-8r 
\end{array}\right]\\
&=\left[\begin{array}{cccc}
0.6250-1.1250r   \\
0.6250-1.1250r \\
0.6250-1.1250r  \\
0.6250-1.1250r
\end{array}\right],\end{aligned}$$
then the associated fuzzy number vector\\
$$\begin{aligned}\tilde{X}=
\left[\begin{array}{cc}
\tilde{x}_{1} \\
\tilde{x}_{2}\\
\end{array}\right]=\left[\begin{array}{cccc}
(0.6250-1.1250r,~ -0.6250+1.1250r)  \\
(0.6250-1.1250r,~ -0.6250+1.1250r)  
\end{array}\right],\end{aligned}$$
is a generalized solution of the FLS (\ref{inconsistent FLS}).
%which implies that $\tilde{x}_{1}$ and $\tilde{x}_{2}$ are fuzzy numbers, therefore $\tilde{X}$ is a strong fuzzy solution. 
%Due to $S^{\scriptsize\textcircled{\tiny\dag}}$ is a nonnegative matrix, see Theorem \ref{panduan strong solution}, we already get a strong fuzzy solution.\\

%Thus $$\begin{aligned}&\left[\begin{array}{cc}\tilde{x}_{11} & \tilde{x}_{12} \\\tilde{x}_{21} & \tilde{x}_{22}\end{array}\right]=\left[\begin{array}{cccc}(\underline{x}_{11},\overline{x}_{11}) & (\underline{x}_{12},\overline{x}_{12}) \\(\underline{x}_{21},\overline{x}_{21}) & (\underline{x}_{22},\overline{x}_{22})\\\end{array}\right]\\&=\left[\begin{array}{cccc}(0.625-1.125r, 1.125r-0.625) & (0.375r, -0.375r)  \\(0.625-1.125r, 1.125r-0.625) & (0.375r, -0.375r)  \\\end{array}\right].\end{aligned}$$ is a strong fuzzy solution of (\ref{(a1)}).\\
$({\rm ii})$ We focus on the following consistent crisp linear system:
\begin{equation}\label{2}
(S^{2})^{*} S X=(S^{2})^{*} Y.\end{equation}
%Then,$$[(S^{2})^{*}S]^{\scriptsize\textcircled{\tiny\dag}}=\left[\begin{array}{llll}2 & 2& 2 & 2 \\2& 2 & 2& 2 \\ 2& 2 & 2 & 2 \\ 2& 2 & 2 & 2\end{array}\right]^{\scriptsize\textcircled{\tiny\dag}}=\left[\begin{array}{llll}0.0313 & 0.0313 & 0.0313 & 0.0313 \\0.0313 & 0.0313 & 0.0313 & 0.0313 \\0.0313 & 0.0313 & 0.0313 & 0.0313 \\0.0313 & 0.0313 & 0.0313 & 0.0313 \end{array}\right].$$
According to \textbf{Proposed method 2}, $$
X=S^{\scriptsize\textcircled{\tiny\dag}}Y=\left[\begin{array}{cccc}
0.1250 & 0.1250 & 0.1250 & 0.1250 \\
0.1250 & 0.1250 & 0.1250 & 0.1250 \\
0.1250 & 0.1250 & 0.1250 & 0.1250 \\
0.1250 & 0.1250 & 0.1250 & 0.1250 
\end{array}\right]$$ is a solution of (\ref{2}), then the associated fuzzy number vector\\
$$\begin{aligned}\tilde{X}=
\left[\begin{array}{cc}
\tilde{x}_{1} \\
\tilde{x}_{2}\\
\end{array}\right]=\left[\begin{array}{cccc}
(0.6250-1.1250r,~ -0.6250+1.1250r)  \\
(0.6250-1.1250r,~ -0.6250+1.1250r)  
\end{array}\right],\end{aligned}$$
is a generalized solution of the FLS (\ref{inconsistent FLS}).
\end{eg}
\begin{con*}
In the above context, we give a solution of the consistent FLS $A\tilde{X}=\tilde{Y}$ and a generalized solution of the inconsistent FLS $A\tilde{X}=\tilde{Y}$. It is natural to ask whether we can obtain the general solution of the FLS $A\tilde{X}=\tilde{Y}$ and this question will be our future research topic.
\end{con*}
\vspace{0.4cm} \noindent {\large\bf Acknowledgements}
This research is supported by the National Natural Science Foundation
of China (No.12001368) and sponsored by Shanghai Sailing Program (No.20YF1433100).

\end{document}